\documentclass[a4paper,11pt]{article}
\usepackage{titling}
\usepackage[english]{babel}
\usepackage[utf8]{inputenc}
\usepackage{graphicx}
\usepackage{amsmath,amssymb,amsfonts,amsthm}
\usepackage{geometry}
\usepackage{float}
\usepackage{enumitem}
\usepackage{caption}
\title{
	The Significance of Proposition~II
	in Galois' \textit{Mémoire}\\[1ex]
	\large The Origin of Galois Automorphisms
}
\author{
	Math Dicker\\
	\small louis.dicker@ziggo.nl
}
\date{\today}
\newtheorem{proposition}{Proposition}
\newtheorem{lemma}{Lemma}

\begin{document}
	
	\maketitle

	\begin{abstract}
		
		In Proposition~II of his manuscript, Galois writes the well-known remark:
		\textbf{\textit{Il y a quelque chose à compléter dans cette démonstration. Je n’ai pas le temps}.}
		Although Galois did not complete the proof, it is possible to reconstruct the essential content of Proposition~II and to supply the missing arguments. As usual, \(V\) denotes the linear form \(ax_1+bx_2+cx_3+\dots\), where \(x_1,x_2,x_3,\dots\) are distinct roots of a separable polynomial. Let \(L\) be the corresponding splitting field over a ground field of characteristic \(0\). On the one hand, Proposition~II concerns the factorization of the minimal polynomial \(g(x)\) of \(V\) over an intermediate field \(M\subseteq L\); on the other hand, it concerns the factorization of the same polynomial into irreducible factors whose coefficients belong to the intermediate fields conjugate to \(M\). It is precisely this latter aspect that constitutes the central theme of Proposition~II.
		
		The notions of \textit{groupe de permutations} and \textit{groupe de substitutions} are of fundamental importance in the \textit{Mémoire}. We provide a characterization of these notions in modern terminology. Associated with every \textit{groupe de permutations} is a polynomial whose coefficients are invariant under the corresponding \textit{groupe de substitutions}. Moreover, a \textit{groupe de permutations} determines a partition of the Galois group, making it possible to factor the minimal polynomial \(g(x)\) into factors whose coefficients belong to the intermediate fields corresponding to the associated \textit{groupes de substitutions}.
		
		Finally, we prove that the substitutions of the Galois group are field automorphisms of the splitting field \(L\). This establishes the connection between Galois' original formulation and the modern formulation of Galois theory.
		
		\begin{figure}[htbp]
			\centering
			\includegraphics[width=0.8\textwidth]{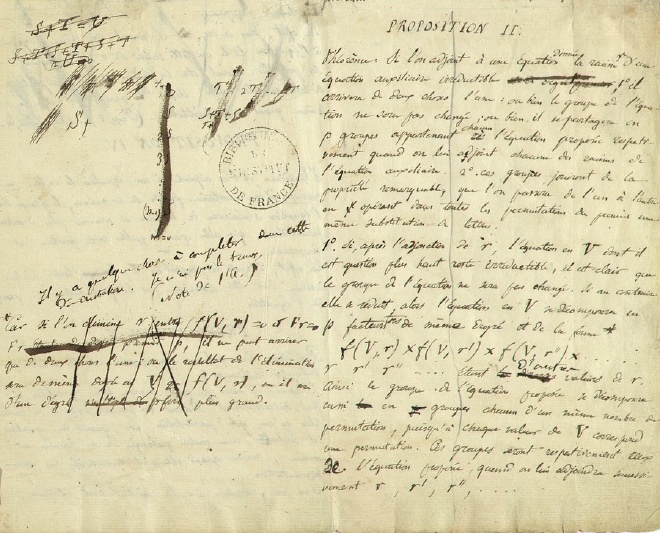}
			\caption*{Left margin: \textit{Il y a quelque chose à compléter dans cette démonstration. Je n’ai pas le temps}.}
			\label{fig:plaatje}
		\end{figure}
		
	\end{abstract}

\section{Introduction}

Proposition~II of the \textit{Mémoire} does not immediately reveal its true significance. In the manuscript, Galois remarks: \textit{Il y a quelque chose à compléter dans cette démonstration. Je n'ai pas le temps.} This remark most likely refers not only to the incompleteness of the proof, but also to the absence of a more detailed explanation of the underlying purpose of Proposition~II. The formula appearing somewhat beyond the middle of the manuscript shows that Proposition~II is essentially concerned with a factorization process. The aim of this article is to demonstrate how this factorization should be understood. Readers who wish to gain a concrete understanding of the underlying ideas are advised to study the detailed example in Section~\ref{sec:voorbeeld} before reading the general proof.

Two notions play a central role throughout the paper: \textit{groupe de permutations} and \textit{groupe de substitutions}. We show that, in Galois' terminology, a \textit{groupe de permutations} corresponds to either a left or a right coset of a group. Consequently, such a set is not necessarily a group in the modern sense, although this is the case in Proposition~I. A \textit{groupe de substitutions}, on the other hand, is always a group. It may be helpful to become familiar with these notions by first considering concrete examples in, for instance, \(S_4\).

Galois speaks almost exclusively of substitutions rather than automorphisms. Adolf Hurwitz adopts the same point of view in his 1890 lectures on Galois theory delivered in Königsberg \cite{Hurwitz1}. At the end of this article, we show that the substitutions occurring in Proposition~I correspond precisely to the automorphisms by means of which Galois theory is formulated today. This establishes a connection between the classical and the modern approaches. For a comprehensive treatment of substitution theory, the reader is referred to Eugen Netto's German or English monograph \cite{netto}. Finally, in Lemma~III Galois proves the existence of certain functions in \(\mathbb{Q}(x)\) that are essential for Proposition~I. We present a direct proof of this result based on an argument closely related to Lagrange interpolation. Hurwitz proves the same result by a different method; nevertheless, both approaches lead to the same functions.

\section{The Galois group}

In Proposition~I, Galois speaks of a \textit{groupe de permutations} and a \textit{groupe de substitutions}; these notions have a specific meaning in Galois' work; see Neumann \cite[pp.~20--23]{neumann}. We shall discuss these concepts in Section~\ref{sec:ps}. The set of permutations of the roots considered by Galois is

\[
\left(
\begin{array}{cccc}
	\theta_1(V_i)&\theta_2(V_i)&\theta_3(V_i)&\theta_4(V_i)
\end{array}
\right),
\]

where \(1 \le i \le m\). Associated with this set of permutations are the sets of substitutions

\[
\left(
\begin{array}{cccc}
	\theta_1(V_i)&\theta_2(V_i)&\theta_3(V_i)&\theta_4(V_i)\\
	\theta_1(V_j)&\theta_2(V_j)&\theta_3(V_j)&\theta_4(V_j)
\end{array}
\right),
\]

where \(1 \le i,j \le m\) and \(i\) is fixed. In principle, this construction yields \(m\) sets, each consisting of \(m\) substitutions. A remarkable fact, already known to Galois, is that each of these \(m\) sets of substitutions coincides with the set

\[
\left(
\begin{array}{cccc}
	\theta_1(V_1)&\theta_2(V_1)&\theta_3(V_1)&\theta_4(V_1)\\
	\theta_1(V_j)&\theta_2(V_j)&\theta_3(V_j)&\theta_4(V_j)
\end{array}
\right),
\]

where \(1 \le j \le m\). It is this property of the set of permutations, which Galois calls a \textit{groupe de permutations}, that implies that the associated set of substitutions is \textbf{closed under composition}, and therefore forms what Galois calls a \textit{groupe de substitutions}. The meaning of the functions \(\theta_1,\theta_2,\theta_3,\theta_4\) will be explained in Section~\ref{sec:ct}. We shall return to all of these notions and their interrelations in the subsequent sections.

\section{On permutations and substitutions}
\label{sec:ps}

\subsection{The notions of \textit{groupe de permutations} and \textit{groupe de substitutions}}

From the principles stated at the beginning of the \textit{Mémoire}, one can infer what Galois means by the notions of \textit{groupe de permutations} and \textit{groupe de substitutions}. Without loss of generality, we consider permutations of the symbols \(1,2,3,\) and \(4\), whereas Galois considers permutations of letters.

A collection of permutations \(p_1,p_2,\dots,p_k\) is called a \textit{groupe de permutations} if it satisfies the following condition.

\bigskip

\noindent
\textbf{Definition.}

A collection of permutations is a \textit{groupe de permutations} if, for every \(1 \le i \le k\), the set of substitutions

\[
\left(
\begin{array}{c}
	p_1\\
	p_1
\end{array}
\right),
\left(
\begin{array}{c}
	p_1\\
	p_2
\end{array}
\right),
\dots,
\left(
\begin{array}{c}
	p_1\\
	p_k
\end{array}
\right)
\]

coincides with the set of substitutions

\[
\left(
\begin{array}{c}
	p_i\\
	p_1
\end{array}
\right),
\left(
\begin{array}{c}
	p_i\\
	p_2
\end{array}
\right),
\dots,
\left(
\begin{array}{c}
	p_i\\
	p_k
\end{array}
\right).
\]

An \textbf{important property} of a substitution is that it remains invariant when both permutations occurring in it are multiplied on the right by the same permutation. Indeed, right multiplication permutes the symbols appearing in both original permutations in exactly the same way, and consequently the substitution itself remains unchanged. This invariance does not, in general, hold under left multiplication.

In symbolic form, for an arbitrary permutation \(x\)

\[
\boxed{
	\left(
	\begin{array}{c}
		p_i\\
		p_j
	\end{array}
	\right)
	\equiv
	\left(
	\begin{array}{c}
		p_ix\\
		p_jx
	\end{array}
	\right)
}
\]\subsection{Characterization of a \textit{groupe de permutations}}

A \textit{groupe de permutations} admits the following characterization.

\begin{lemma}
	\label{lem:een}
	A set of permutations \(P=\{p_1,p_2,\dots,p_k\}\) is a \textbf{\textit{groupe de permutations}} if and only if \(P\) is a \textbf{right coset of a group}, if and only if \(P\) is a \textbf{left coset of a group}. In the former case, if \(P\) is a right coset of a group \(H\), then \(H\) is the corresponding \textbf{\textit{groupe de substitutions}}. In the latter case, if \(P\) is a left coset \(aH\), then \(aHa^{-1}\) is the corresponding \textbf{\textit{groupe de substitutions}}.
\end{lemma}

\begin{proof}[\normalfont\bfseries Proof]
	
	The last two equivalences are immediate. Indeed, if \(H\) is a group and \(a\) is a permutation, then
	\[
	aH=(aHa^{-1})a
	\]
	and
	\[
	Ha=a(a^{-1}Ha).
	\]
	
	First suppose that \(P\) contains the identity permutation \(1\). We then show that
	
	\[
	\{\text{\(P\) is a \textit{groupe de permutations} containing the identity}\}
	\Longleftrightarrow
	\{\text{\(P\) is a group}\}.
	\]
	
	Indeed, let \(x,y\in P\). Since \(P\) is a \textit{groupe de permutations}, there exists \(z\in P\) such that
	
	\[
	\left(
	\begin{array}{c}
		1\\
		x
	\end{array}
	\right)
	=
	\left(
	\begin{array}{c}
		y\\
		z
	\end{array}
	\right)
	=
	\left(
	\begin{array}{c}
		1\\
		zy^{-1}
	\end{array}
	\right).
	\]
	
	Hence \(x=zy^{-1}\), and therefore \(xy=z\in P\). Thus \(P\) is closed under composition.
	
	Conversely,
	
	\[
	\left(
	\begin{array}{c}
		p_i\\
		p_j
	\end{array}
	\right)
	=
	\left(
	\begin{array}{c}
		1\\
		p_jp_i^{-1}
	\end{array}
	\right),
	\]
	
	and since \(P\) is a group, \(p_jp_i^{-1}\in P\). Consequently, the defining condition of a \textit{groupe de permutations} is satisfied.
	
	Now suppose that \(P\) does not contain the identity permutation. Consider the set \(Pp_1^{-1}\). This is again a \textit{groupe de permutations}, but now containing the identity. By the previous argument it is therefore a group, say \(H\). Consequently,
	
	\[
	P=Hp_1,
	\]
	
	so that \(P\) is a right coset of \(H\).
	
	Conversely, suppose that \(P=Hp_1\) is a right coset of the group
	\(H=\{h_1,h_2,\dots,h_r\}\). Then
	
	\[
	\left(
	\begin{array}{c}
		h_ip_1\\
		h_jp_1
	\end{array}
	\right)
	=
	\left(
	\begin{array}{c}
		h_i\\
		h_j
	\end{array}
	\right),
	\]
	
	showing that \(Hp_1\) is a \textit{groupe de permutations}.
	
	Finally, suppose that \(P=Ha\) is a right coset of \(H\). Then
	
	\[
	\left(
	\begin{array}{c}
		h_ia\\
		h_ja
	\end{array}
	\right)
	=
	\left(
	\begin{array}{c}
		h_i\\
		h_j
	\end{array}
	\right),
	\]
	
	so that the associated substitutions are precisely those arising from \(H\).
	
	If, on the other hand, \(P=aH\) is a left coset of \(H\), then
	
	\[
	\left(
	\begin{array}{c}
		ah_i\\
		ah_j
	\end{array}
	\right)
	=
	\left(
	\begin{array}{c}
		ah_ia^{-1}\\
		ah_ja^{-1}
	\end{array}
	\right),
	\]
	
	and hence the associated substitutions are exactly those arising from the conjugate group \(aHa^{-1}\).
	
\end{proof}

We also record the following immediate consequence.

\begin{lemma}
	\label{lem:twee}
	A set of permutations \(P=\{p_1,p_2,\dots,p_k\}\) containing the identity permutation is a \textbf{\textit{groupe de permutations}} if and only if \(P\) is a group.
\end{lemma}

If \(P\) is a \textbf{\textit{groupe de permutations}}, then the corresponding set of substitutions is closed under composition and therefore forms a group. Indeed, let

\[
S=
\left(
\begin{array}{c}
	p_1\\
	p_r
\end{array}
\right),
\qquad
T=
\left(
\begin{array}{c}
	p_1\\
	p_s
\end{array}
\right),
\qquad
S^{-1}=
\left(
\begin{array}{c}
	p_1\\
	p_t
\end{array}
\right),
\]

where \(1\le r,s,t\le k\). Then

\[
\left(
\begin{array}{c}
	p_t\\
	p_s
\end{array}
\right)
=
\left(
\begin{array}{c}
	p_1\\
	p_s
\end{array}
\right)
\left(
\begin{array}{c}
	p_t\\
	p_1
\end{array}
\right)
=
T\bigl((S^{-1})^{-1}\bigr)
=
TS,
\]

which therefore also belongs to the set. We thus obtain the following result.

\begin{lemma}
	\label{lem:drie}
	The \textbf{\textit{groupe de substitutions}} associated with a \textbf{\textit{groupe de permutations}} is a group.
\end{lemma}

It follows that the substitutions constituting the Galois group in Proposition~I form a group.

\section{Context and identification of the roots \(V_i\), permutations and coordinates}
\label{sec:ct}

We follow the edition of the \textit{Mémoire} published by Liouville in 1846. In this section we collect the results from the \textit{Mémoire} that will be used throughout the remainder of the paper.

\subsection{Context}

\begin{enumerate}
	
	\item
	We restrict our attention to the case of an irreducible polynomial of degree four. The general case is treated \emph{mutatis mutandis} in exactly the same way. Let \(f\in\mathbb{Q}[x]\) be an irreducible polynomial with four distinct roots \(x_1,x_2,x_3,x_4\), and let \(L=\mathbb{Q}(x_1,x_2,x_3,x_4)\) be its splitting field. According to Lemma~II of the \textit{Mémoire}, there exist integers \(a,b,c,\) and \(d\) such that all \(24\) permutations of \(ax_1+bx_2+cx_3+dx_4\) are distinct. We denote these \(24\) values by \(V_1,V_2,V_3,\dots,V_{24}\), where \(V_1=ax_1+bx_2+cx_3+dx_4\). Galois also denotes \(V_1\) simply by \(V\).
	
	\item
	According to Proposition~I of the \textit{Mémoire}, there exist rational functions \(\theta_1,\theta_2,\theta_3,\theta_4\in\mathbb{Q}(x)\) such that
	\[
	V_i=a\theta_1(V_i)+b\theta_2(V_i)+c\theta_3(V_i)+d\theta_4(V_i).
	\]
	
	\item
	The existence of the functions \(\theta_1,\theta_2,\theta_3,\theta_4\) implies that \(V_1\) is a primitive element of the splitting field \(L\). Consequently, for each \(V_i\) there exists a unique polynomial \(F_i\in\mathbb{Q}[x]\) whose degree is less than \([L:\mathbb{Q}]\) such that \(V_i=F_i(V_1)\).
	
	\item
	Let \(M\) be an intermediate field satisfying \(\mathbb{Q}\subseteq M\subseteq L\).
	
	\item
	We shall use the following notation. Let \(h(x)\in\mathbb{Q}[x]\) denote the polynomial whose roots are precisely the \(24\) values \(V_i\). Let \(m(x)\in M[x]\) denote the minimal polynomial of \(V_1\) over \(M\), and let \(g(x)\in\mathbb{Q}[x]\) denote the minimal polynomial of \(V_1\) over \(\mathbb{Q}\). Then \(m(x)\) divides \(g(x)\), and \(g(x)\) divides \(h(x)\). Finally, we renumber the values
	\(
	V_1,V_2,\dots,V_{24}
	\)
	so that the first \(m\) values
	\(
	V_1,V_2,\dots,V_m
	\)
	are precisely the roots of \(m(x)\).
	
\end{enumerate}

\subsection{Identification of the roots \(V_i\) with permutations and coordinates}

By Lemma~II, the \(24\) values obtained by applying the \(24\) permutations \(\alpha\beta\gamma\delta\) to \(V_1\) are pairwise distinct. Consequently, the corresponding values of
\(a\theta_\alpha(V_1)+b\theta_\beta(V_1)+c\theta_\gamma(V_1)+d\theta_\delta(V_1)\)
are likewise all distinct. The permutation \(1234\) associated with \(V_1\) is multiplied \underline{on the right} by the permutation \(\alpha\beta\gamma\delta\); throughout this paper, permutations are composed from right to left. The convention adopted by Galois himself deserves further historical investigation; we conjecture that he composed permutations from left to right. Each of the \(24\) values can therefore be associated with a unique \textbf{permutation}. We shall also refer to this permutation as the \textbf{coordinates} of the corresponding value. The permutation associated with \(V_i\) will be denoted by \(p_i\).

\section{The functions \(\phi_{\alpha\beta\gamma\delta}\), \(F_i\) and the substitution
	\(
	\left(
	\begin{array}{c}
		p_1\\
		p_i
	\end{array}
	\right)
	\)
}

\begin{enumerate}
	
	\item
	
	To every permutation \(\alpha\beta\gamma\delta\in S_4\) we associate the polynomial
	\[
	\phi_{\alpha\beta\gamma\delta}(v)
	=
	a\theta_\alpha(v)
	+b\theta_\beta(v)
	+c\theta_\gamma(v)
	+d\theta_\delta(v).
	\]
	
	Here, \(\theta_\alpha(V_i)\) denotes the coordinate of \(V_i\) in position \(\alpha\), \(\theta_\beta(V_i)\) the coordinate in position \(\beta\), and so forth. We then have
	
	\[
	\phi_{p_j}(V_i)
	=
	F_j(V_i)
	=
	V_{p_ip_j},
	\]
	
	where \(p_ip_j\) denotes the product of the permutations \(p_i\) and \(p_j\), with \(1\le i,j\le m\). Indeed, if \(V_j=F_j(V_1)\) and
	\(V_j=a\theta_\alpha(V_1)+b\theta_\beta(V_1)+c\theta_\gamma(V_1)+d\theta_\delta(V_1)\),
	where \(p_j=\alpha\beta\gamma\delta\), then Lemma~I of the \textit{Mémoire} implies that
	
	\[
	F_j(V_i)
	=
	a\theta_\alpha(V_i)
	+b\theta_\beta(V_i)
	+c\theta_\gamma(V_i)
	+d\theta_\delta(V_i),
	\]
	
	for \(i=1,2,\dots,m\). Thus, \(F_j\) transforms the coordinates of \textbf{every} \(V_i\) according to a \textbf{fixed permutation pattern} determined solely by right multiplication of \(p_i\) by \(p_j=\alpha\beta\gamma\delta\).
	
	\item
	
	The substitution	
	\(
	\begin{pmatrix}
		p_1\\
		p_j
	\end{pmatrix}
	\)
		acts by left multiplication:
		\(
	\begin{pmatrix}
		p_1\\
		p_j
	\end{pmatrix}(V_i)
	=
	V_{p_jp_i}.
	\)
	
\end{enumerate}
\newpage
 \newpage
 \section{An algebraic interpretation of Proposition~II}
 
 We begin by reproducing Liouville's transcription of Proposition~II and first explain it from an algebraic point of view. 
 
 \begin{figure}[H]
 	\centering
 	\includegraphics[width=1\textwidth]{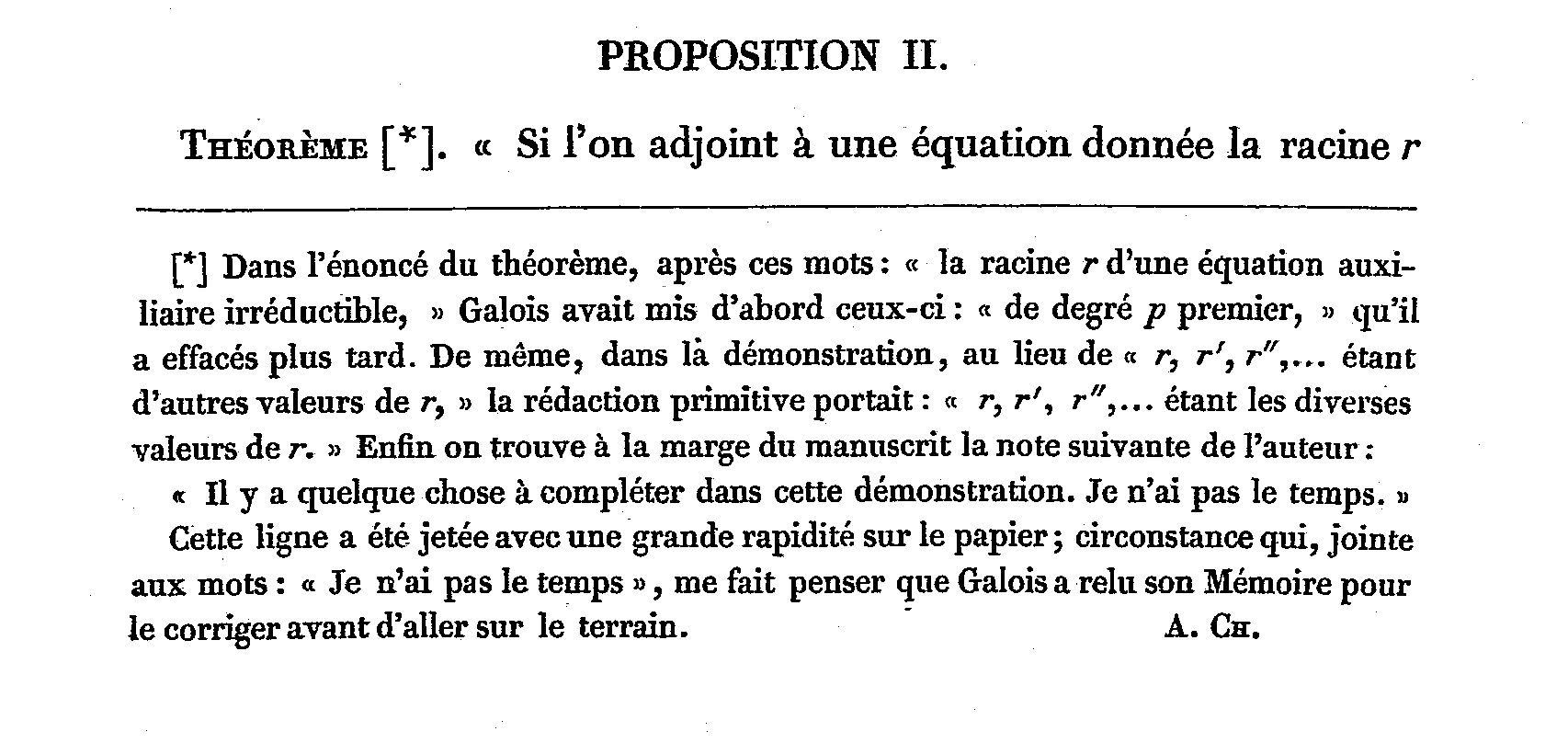}
 \end{figure}
 
 \begin{figure}[H]
 	\centering
 	\includegraphics[width=1\textwidth]{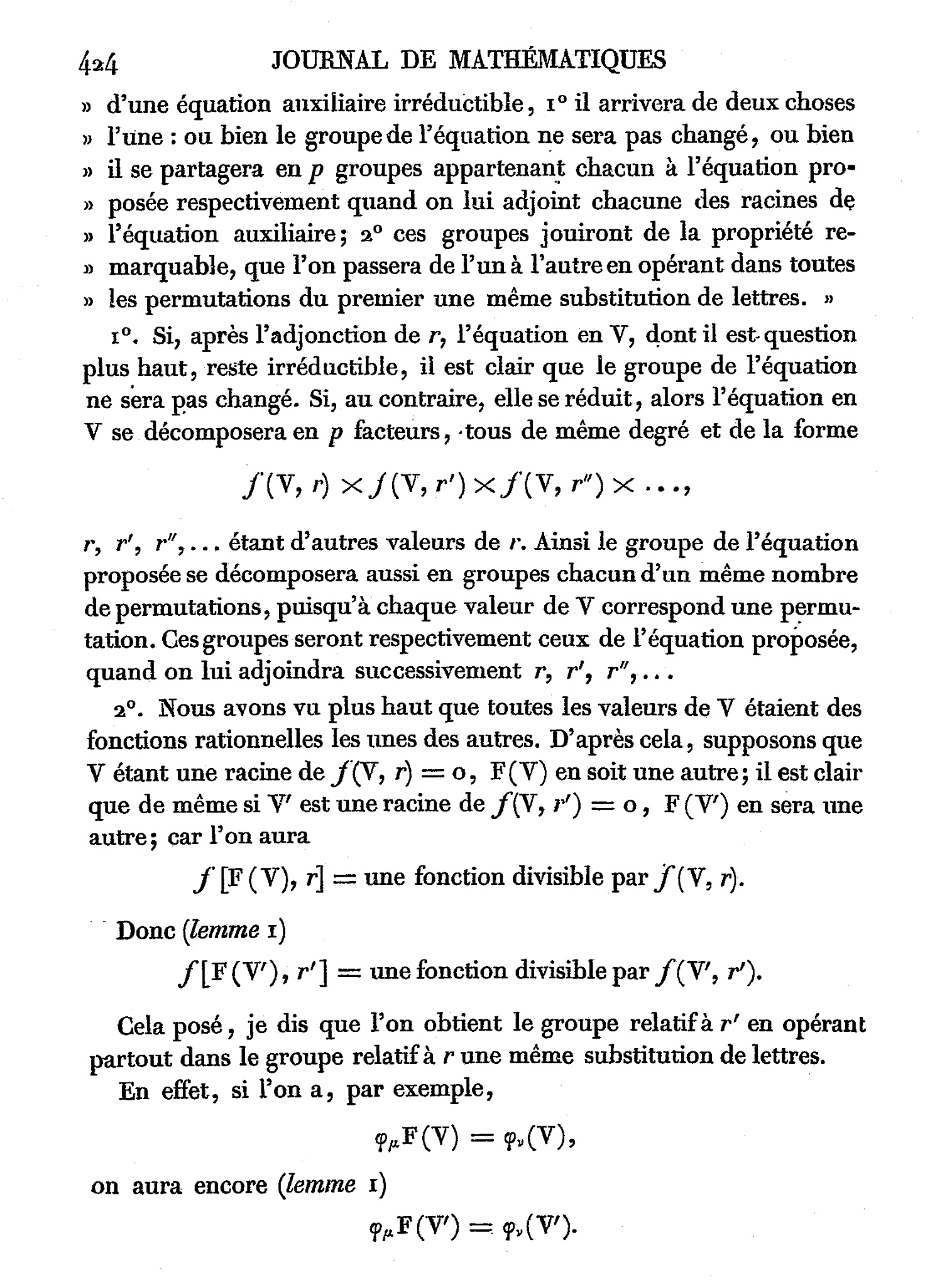}
 \end{figure}
 
 \begin{figure}[H]
 	\centering
 	\includegraphics[width=1\textwidth]{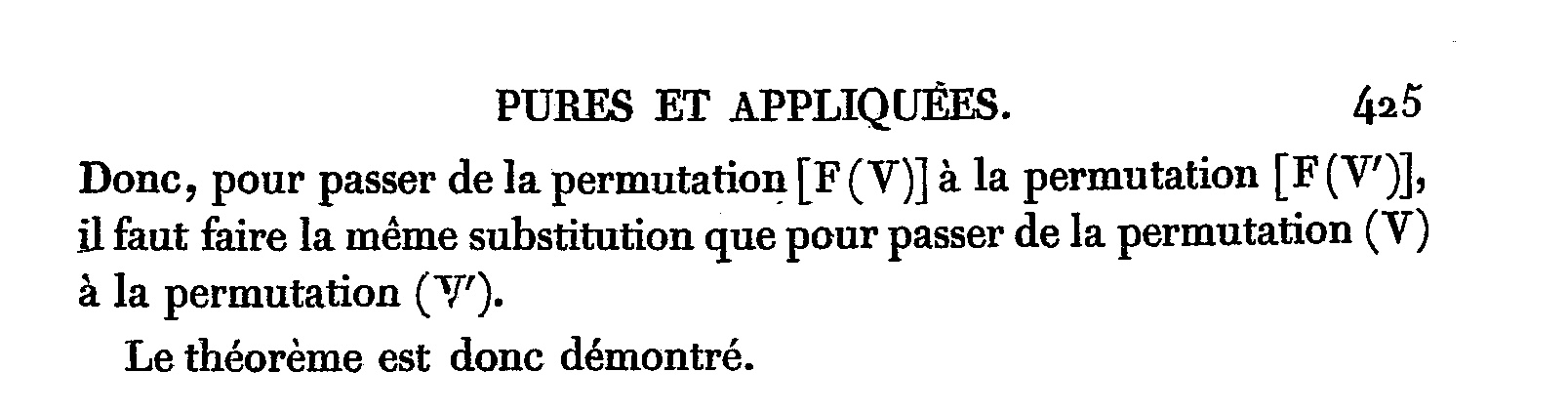}
 	\caption*{Liouville's 1846 transcription of the \textit{Mémoire}.}
 	\label{fig:plaatje}
 \end{figure}
 
 For the sake of clarity, we restrict ourselves to the following situation. The general case can be treated in exactly the same way, and the essential idea of the argument remains unchanged. We assume that the degree of the splitting field \(L\) over \(\mathbb{Q}\) is \(24\). Galois adjoins to \(\mathbb{Q}\) a root \(r\) of an irreducible polynomial \(z(x)\in\mathbb{Q}[x]\), thereby obtaining the field \(\mathbb{Q}[r]=\mathbb{Q}(r)\). We assume that \(z(x)\) has degree \(4\) with roots \(r,r',r'',r'''\). The assumption that \(z(x)\) has degree \(4\) is made only to facilitate the discussion in the next section. The same reasoning remains valid if one replaces the ground field \(\mathbb{Q}\) by an intermediate field \(\mathbb{Q}\subseteq M\subseteq L\) and considers an irreducible polynomial \(z(x)\in M[x]\) of arbitrary degree.
 
 Let \(f(v,r)\) denote the minimal polynomial of \(V\) over \(\mathbb{Q}(r)\). Since \(V\) is a primitive element of \(L\), every other root of this polynomial is of the form \(F(V)\), where \(F\in\mathbb{Q}[x]\). By Lemma~I of the \textit{Mémoire}, we therefore have
  \begin{equation}
 	\label{eq:kern}
 	f(F(v),r)=f(v,r)s(v,r).
 \end{equation}
 
 In equation~(\ref{eq:kern}), all polynomials belong to \(\mathbb{Q}(r)[x]\). The crucial step is now to apply Lemma~I of the \textit{Mémoire} to the coefficients of equation~(\ref{eq:kern}). Since the coefficients of equal powers of \(v\) on both sides of the equation coincide and belong to \(\mathbb{Q}(r)\), Lemma~I yields
  \begin{equation}
 	\label{eq:kern2}
 	f(F(v),r')=f(v,r')s(v,r').
 \end{equation}
 
 Here all polynomials belong to \(\mathbb{Q}(r')[x]\). The same argument applies to the conjugate roots \(r''\) and \(r'''\). Consequently, if \(V'\) is a root of \(f(v,r')\), then \(F(V')\) is also a root of \(f(v,r')\).
 
 We now turn to the question of how Galois' symbols \(\phi_\mu\) and \(\phi_\nu\) should be interpreted, insofar as such an interpretation is possible. Galois applies Lemma~I, indicating that these symbols represent functions. At the end of Proposition~II he further states that the two substitutions
 
 \[
 \begin{pmatrix}
 	F(V)\\
 	F(V')
 \end{pmatrix}
 \qquad\text{and}\qquad
 \begin{pmatrix}
 	V\\
 	V'
 \end{pmatrix}
 \]
 
 must be identical. Such an identity can only arise by multiplying both permutations occurring in a substitution on the right by one and the same permutation.
 
 These two observations naturally lead us to identify Galois' functions \(\phi_\mu\) and \(\phi_\nu\) with the previously introduced functions \(\phi_{\alpha\beta\gamma\delta}\), where the permutations are \(\mu\) and \(\nu\), respectively. With this interpretation, the proof proceeds naturally. Indeed, for every permutation \(\mu\) there exists a permutation \(\nu\) such that
  \[
 \phi_\mu(F(V))=\phi_\nu(V),
 \]
 
 and, by Lemma~I applied to the minimal polynomial of \(V\) over \(\mathbb{Q}\), which has degree \(24\), the same identity also holds after replacing \(V\) by \(V'\). 
 Consequently, 
 \[
 \begin{pmatrix}
 	F(V)\\
 	F(V')
 \end{pmatrix}
 =
 \begin{pmatrix}
 	\phi_\mu(F(V))\\
 	\phi_\mu(F(V'))
 \end{pmatrix}
 =
 \begin{pmatrix}
 	\phi_\nu(V)\\
 	\phi_\nu(V')
 \end{pmatrix}
 =
 \begin{pmatrix}
 	V\\
 	V'
 \end{pmatrix}.
 \]
 
 The first and the last equalities follow from the fact that the application of \(\phi_\mu\) and \(\phi_\nu\) leaves the corresponding substitutions unchanged. The final conclusion is therefore exactly the one reached by Galois.
 
 Liouville's transcription differs slightly from the original manuscript at this point. Galois uses the symbols \(p\) and \(n\) rather than permutations. Liouville also wrote an extensive commentary on Proposition~II. For a detailed discussion, the reader is referred to the book by Jesper Lützen \cite{lutzen}.
 
 \begin{figure}[H]
 	\centering
 	\includegraphics[width=1\textwidth]{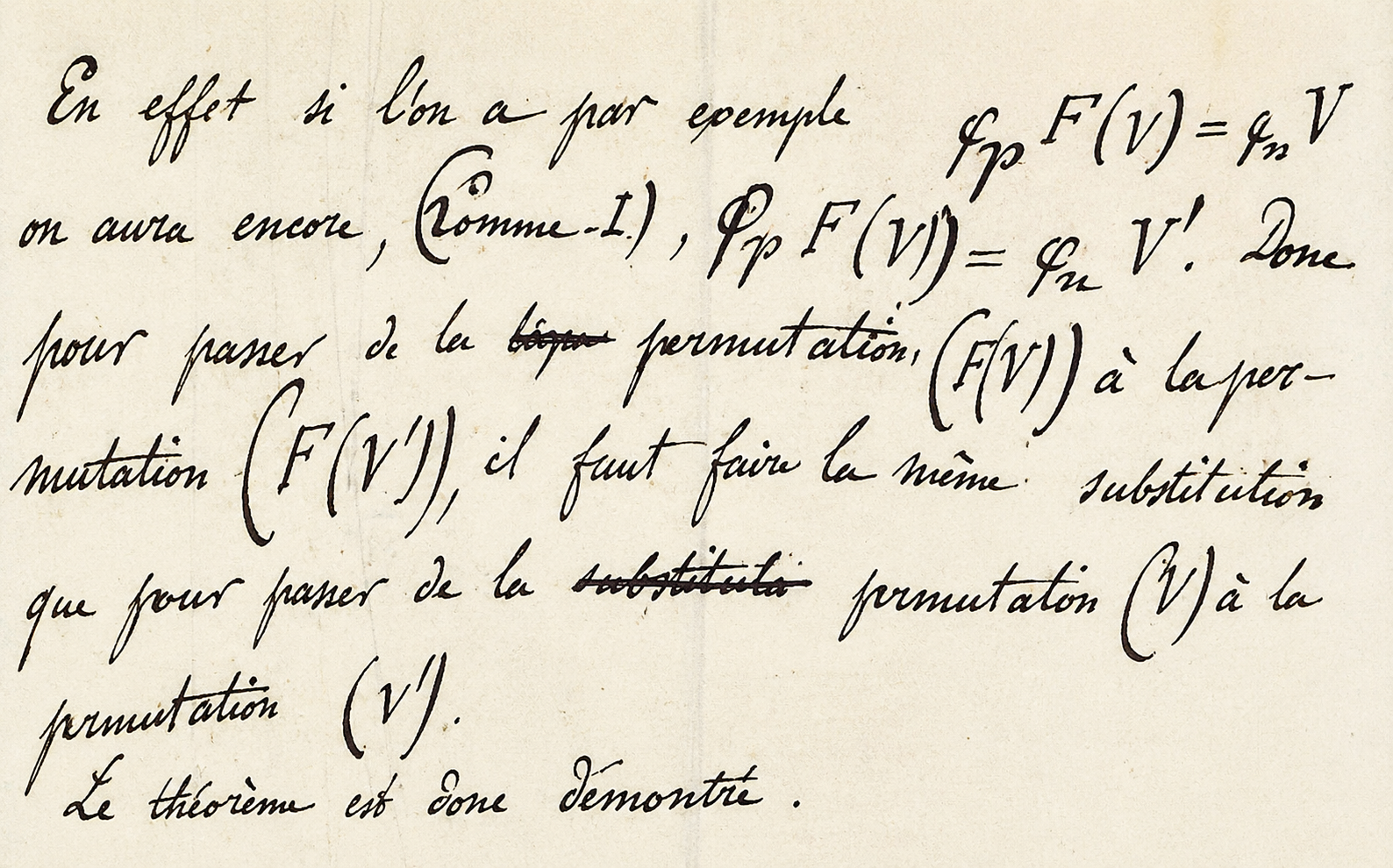}
 	\caption*{The corresponding passage from Galois' manuscript.}
 	\label{fig:plaatje}
 \end{figure}

\section{A clarification of Proposition~II in terms of \textit{groupes de permutations} and \textit{groupes de substitutions}}

\subsection{A further interpretation of Proposition~II using the notions of \textit{groupe de permutations} and \textit{groupe de substitutions}}

Let \(M\subseteq L\), and let \(m(x)\) be the minimal polynomial of \(V_1\) over \(M\), with roots
\[
\{V_1,V_2,\dots,V_m\}
=
\{V_{\pi_1},V_{\pi_2},\dots,V_{\pi_m}\}.
\]
By the definition of the polynomials \(F_j\), we have

\begin{equation}
\label{eq:ge10}
\{V_1,V_2,\dots,V_m\}
=
\{F_1(V_1),F_2(V_1),\dots,F_m(V_1)\}.
\end{equation}

Let \(V'=V_j\) in Proposition~II, where \(V'\) is a root of \(f(v,r')\). The roots of \(f(v,r')\) are therefore

\begin{equation}
\label{eq:ge11}
\{F_1(V'),F_2(V'),\dots,F_m(V')\}
=
\{F_1(V_j),F_2(V_j),\dots,F_m(V_j)\}
=
\{V_{\pi_j\pi_1},V_{\pi_j\pi_2},\dots,V_{\pi_j\pi_m}\}.
\end{equation}

After identifying the values \(V_i\) with their corresponding permutations, equation~(\ref{eq:ge11}) shows that the roots of the polynomial \(f(v,r')\) are precisely the permutations belonging to the \textbf{left coset}

\[
\{\pi_j\pi_1,\pi_j\pi_2,\dots,\pi_j\pi_m\}.
\]

By Lemma~\ref{lem:een}, this left coset is a \textit{groupe de permutations}, and the corresponding \textit{groupe de substitutions} is the group

\[
\pi_j\{\pi_1,\pi_2,\dots,\pi_m\}\pi_j^{-1}.
\]

It follows that the coefficients of the polynomial \(f(v,r')\) belong to the field \(\pi_j(M)\), which is conjugate to \(M\). In Galois' notation this field is \(\mathbb{Q}(r')\). Clearly, \(f(v,r')\) divides the minimal polynomial \(g(x)\in\mathbb{Q}[x]\) of \(V_1\) over \(\mathbb{Q}\). Proceeding in this way, one obtains a complete factorization of \(g(x)\) over the fields conjugate to \(M\). We shall shortly see how \(g(x)\) factors over \(M\) itself; this is accomplished by means of right cosets.

Galois also explains how one passes from the roots of \(f(v,r)\) to those of \(f(v,r')\), namely by means of the substitution

\[
\begin{pmatrix}
F(V)\\
F(V')
\end{pmatrix}
=
\begin{pmatrix}
V\\
V'
\end{pmatrix}.
\]

In the present situation,

\[
\begin{pmatrix}
V\\
V'
\end{pmatrix}
=
\begin{pmatrix}
V_1\\
V_j
\end{pmatrix}
=
\begin{pmatrix}
\pi_1\\
\pi_j
\end{pmatrix}.
\]

Applying this substitution to the set

\[
\{V_{\pi_1},V_{\pi_2},\dots,V_{\pi_m}\}
\]

produces, after identifying the values with their associated permutations, precisely the left coset described above.

The following proposition provides a modern formulation of the principal idea underlying Proposition~II.

For every permutation \(\sigma\in S_4\), we define \(V_\sigma\) to be the image of \(V_1\) under the substitution

\[
\begin{pmatrix}
1\\
\sigma
\end{pmatrix},
\]

that is,

\[
V_\sigma=
\begin{pmatrix}
1\\
\sigma
\end{pmatrix}(V_1).
\]

\begin{proposition}
\label{prop:nuttig}
Let \(W\) be a \textbf{\textit{groupe de permutations}}, and let \(M\) be the field corresponding to the associated \textbf{\textit{groupe de substitutions}}. Then the polynomial

\[
\prod_{\sigma\in W}\bigl(v-V_\sigma\bigr)
\]

belongs to \(M[v]\), and it is the minimal polynomial over \(M\) of each of its roots.
\end{proposition}

\begin{proof}[\normalfont\bfseries Proof]

The polynomial is invariant under the associated \textbf{\textit{groupe de substitutions}} and therefore belongs to \(M[v]\). Its degree is equal to \([L:M]\), the degree of the extension \(L/M\). It follows that this polynomial is the minimal polynomial over \(M\) of each of its roots.

\end{proof}

The proposition can be formulated more generally for an arbitrary \textbf{primitive element} \(\alpha\).

\begin{proposition}
\label{prop:nut}
Let \(W\) be a \textbf{\textit{groupe de permutations}}, let \(M\) be the field corresponding to the associated \textbf{\textit{groupe de substitutions}}, and let \(\alpha\) be a primitive element of \(L\). Then the polynomial

\[
\prod_{\sigma\in W}
\left(
v-
\begin{pmatrix}
	1\\
	\sigma
\end{pmatrix}(\alpha)
\right)
\]

belongs to \(M[v]\), and it is the minimal polynomial over \(M\) of each of its roots.
\end{proposition}

\begin{proof}[\normalfont\bfseries Proof]

Let

\[
\alpha=\delta(x_1,x_2,x_3,x_4)=\mathcal{F}(V_1),
\]

where \(\mathcal{F}\in\mathbb{Q}[x]\) and \(\delta\in\mathbb{Q}[x_1,x_2,x_3,x_4]\). Then

\[
\alpha=
\delta\bigl(\theta_1(V_1),\theta_2(V_1),\theta_3(V_1),\theta_4(V_1)\bigr)
=\mathcal{F}(V_1),
\]

and Lemma~I of the \textit{Mémoire} implies that

\[
\begin{pmatrix}
	V_1\\
	V_i
\end{pmatrix}
(\alpha)
=
\mathcal{F}
\left(
\begin{pmatrix}
	V_1\\
	V_i
\end{pmatrix}
(V_1)
\right).
\]

Hence,

\[
\prod_{\sigma\in W}
\left(
v-
\begin{pmatrix}
	1\\
	\sigma
\end{pmatrix}(\alpha)
\right)
=
\prod_{\sigma\in W}
\bigl(v-\mathcal{F}(V_\sigma)\bigr).
\]

This polynomial is invariant under the associated \textbf{\textit{groupe de substitutions}}, and its degree is equal to \([L:M]\), the degree of the extension \(L/M\). Here \(\alpha\) is represented by the polynomial \(\delta\) in the variables \(x_1,x_2,x_3,x_4\). If another polynomial \(\epsilon\in\mathbb{Q}[x_1,x_2,x_3,x_4]\) satisfies

\[
\delta(x_1,x_2,x_3,x_4)
=
\epsilon(x_1,x_2,x_3,x_4),
\]

then Lemma~I of the \textit{Mémoire} yields

\[
\delta(\theta_1(V_i),\theta_2(V_i),\theta_3(V_i),\theta_4(V_i))
=
\epsilon(\theta_1(V_i),\theta_2(V_i),\theta_3(V_i),\theta_4(V_i)).
\]

See also Section~\ref{sec:subaut}.

\end{proof}

\subsection{The significance of Proposition~II}

Proposition~\ref{prop:nuttig} above provides a clear understanding of the significance of Proposition~II of the \textit{Mémoire}.

Galois partitions \(S_4\) into left cosets of a group. This yields a factorization of the minimal polynomial \(g(x)\) of \(V_1\) over \(\mathbb{Q}\) into factors defined over the \underline{conjugate fields} corresponding to the groups conjugate to the given group. The \textit{groupe de substitutions} associated with a left coset of a group is precisely the corresponding conjugate group.

If, instead, one partitions \(S_4\) into right cosets of a group, one obtains a factorization of \(g(x)\) over the \underline{field corresponding to that group}, since all right cosets of a group have the same \textit{groupe de substitutions}.

The fact that both left and right cosets are \textit{groupes de permutations} plays a fundamental role in Proposition~II. The notions of \textit{groupe de permutations} and \textit{groupe de substitutions} are therefore essential to Galois' way of thinking.
 		
\section{Proposition~II illustrated by a concrete Example}
\label{sec:voorbeeld}

To illustrate Proposition~II, we consider the polynomial \(x^3-2\), whose roots are
\[
x_1=\alpha=\sqrt[3]{2},\qquad
x_2=\omega\alpha,\qquad
x_3=\omega^2\alpha,
\]
where \(\omega\) is a primitive cube root of unity. By the Fundamental Theorem of Galois Theory, we obtain the following correspondence between subgroups and intermediate fields:

\[
\begin{array}{|c|c|c|}
	\hline
	\text{Subgroup }H & L^{H} & [L^{H}:\mathbb{Q}]\\
	\hline
	S_{3} &
	\mathbb{Q} &
	1\\
	\hline
	A_{3} &
	\mathbb{Q}(\omega)=\mathbb{Q}(\sqrt{-3}) &
	2\\
	\hline
	\langle(23)\rangle &
	\mathbb{Q}\!\left(\sqrt[3]{2}\right) &
	3\\
	\hline
	\langle(13)\rangle &
	\mathbb{Q}\!\left(\omega\sqrt[3]{2}\right) &
	3\\
	\hline
	\langle(12)\rangle &
	\mathbb{Q}\!\left(\omega^{2}\sqrt[3]{2}\right) &
	3\\
	\hline
	\{e\} &
	L=\mathbb{Q}\!\left(\sqrt[3]{2},\omega\right) &
	6\\
	\hline
\end{array}
\]

\begin{enumerate}
	
	\item \textbf{Left cosets.}
	
	We consider the subgroup \(\{123,132\}\) together with its associated \textbf{left cosets}. The corresponding fixed field is
	\(\mathbb{Q}(\sqrt[3]{2})\), since this subgroup leaves \(x_1\) invariant.
	
	\begin{equation}
		\label{eq:tabel221}
		\begin{array}{|c|c|c|}
			\hline
			123 & 231 & 312\\
			132 & 213 & 321\\
			\hline
		\end{array}
	\end{equation}
	
	Although this is a very simple example, it already exhibits the essence of Proposition~II. The set
	
	\[
	\{V_1,V_2,\dots,V_6\}
	=
	\{123,132,231,213,312,321\}
	\]
	
	is partitioned according to the left cosets into the three sets
	
	\[
	W_1=\{123,132\},\qquad
	W_2=\{231,213\},\qquad
	W_3=\{312,321\}.
	\]
	
	To each \textit{groupe de permutations} \(W_i\) we associate the polynomial
	
	\[
	\prod_{\sigma\in W_i}\bigl(v-\sigma(V_1)\bigr),
	\]
	
	where
	
	\[
	V_1=x_1+3x_2+5x_3.
	\]
	
	With the coefficients \(1,3,\) and \(5\), the six permutations of \(V_1\) are pairwise distinct. By Proposition~II, the coefficients of these polynomials belong to the fields corresponding to the associated \textit{groupes de substitutions}. The sets \(W_1,W_2,\) and \(W_3\) yield the following polynomials:
	
	\[
	\begin{aligned}
		\prod_{s\in\{123,132\}}\bigl(v-s(V_1)\bigr)
		&=
		v^2+6\sqrt[3]{2}\,v+12\sqrt[3]{4},\\
		\prod_{s\in\{231,213\}}\bigl(v-s(V_1)\bigr)
		&=
		v^2+6\omega\sqrt[3]{2}\,v+12\omega^2\sqrt[3]{4},\\
		\prod_{s\in\{312,321\}}\bigl(v-s(V_1)\bigr)
		&=
		v^2+6\omega^2\sqrt[3]{2}\,v+12\omega\sqrt[3]{4}.
	\end{aligned}
	\]
	
	The coefficients of these three polynomials belong, respectively, to
	
	\[
	\mathbb{Q}(\alpha),\qquad
	\mathbb{Q}(\omega\alpha),\qquad
	\mathbb{Q}(\omega^2\alpha).
	\]
	
	The essential point is that the corresponding \textit{groupes de substitutions} are
	
	\[
	\{123,132\},\qquad
	\{123,321\},\qquad
	\{123,213\}.
	\]
	
	These three groups are conjugate to one another and therefore correspond to the three conjugate fields. This is precisely the phenomenon described by Galois in Proposition~II: the minimal polynomial of \(V_1\) over \(\mathbb{Q}\) factors into polynomials whose coefficients belong to mutually conjugate intermediate fields.
	
	\item \textbf{Right cosets.}
	
	For right cosets, a different phenomenon occurs. The resulting polynomials all have their coefficients in the fixed field of the subgroup with which one starts. We again consider the subgroup
	
	\[
	H=\{123,132\},
	\]
	
	together with its associated \textbf{right cosets}. The corresponding fixed field is
	
	\[
	\mathbb{Q}(\sqrt[3]{2}),
	\]
	
	since \(H\) fixes \(x_1\). Observe that these right cosets are likewise \textit{groupes de permutations}. The crucial distinction is that the associated \textit{groupes de substitutions} are all equal to the subgroup \(H\). Consequently, all coefficients lie in the same intermediate field \(\mathbb{Q}(\sqrt[3]{2})\).
	
	\begin{equation}
		\label{eq:tabel2222}
		\begin{array}{|c|c|c|}
			\hline
			123 & 231 & 312\\
			132 & 321 & 213\\
			\hline
		\end{array}
	\end{equation}
	
	This yields the following polynomials, all belonging to
	\(\mathbb{Q}(\sqrt[3]{2})[v]\):
	
	\[
	\begin{aligned}
		\prod_{s\in\{123,132\}}\bigl(v-s(V_1)\bigr)
		&=
		v^2+6\sqrt[3]{2}\,v+12\sqrt[3]{4},\\[1ex]
		\prod_{s\in\{231,321\}}\bigl(v-s(V_1)\bigr)
		&=
		v^2-6\sqrt[3]{2}\,v+12\sqrt[3]{4},\\[1ex]
		\prod_{s\in\{312,213\}}\bigl(v-s(V_1)\bigr)
		&=
		v^2+12\sqrt[3]{4}.
	\end{aligned}
	\]
	
	Multiplying these three factors gives
	
	\[
	\begin{aligned}
		&(v^2+6\sqrt[3]{2}\,v+12\sqrt[3]{4})
		(v^2-6\sqrt[3]{2}\,v+12\sqrt[3]{4})
		(v^2+12\sqrt[3]{4})\\
		&\qquad
		=
		v^6+6912.
	\end{aligned}
	\]
	
	The polynomial on the right-hand side is the minimal polynomial of \(V_1\) over
	\(\mathbb{Q}\).
	
	\item \textbf{The alternating group \(A_3\).}
	
	Finally, we consider the normal subgroup
	
	\[
	A_3=\{123,231,312\},
	\]
	
	whose corresponding fixed field is \(\mathbb{Q}(\omega)\). Since \(A_3\) is a normal subgroup of \(S_3\), it suffices to consider its left cosets.
	
	\begin{equation}
		\label{eq:tabel2223}
		\begin{array}{|c|c|}
			\hline
			123 & 132\\
			231 & 321\\
			312 & 213\\
			\hline
		\end{array}
	\end{equation}
	
	\[
	\begin{array}{|c|c|}
		\hline
		W &
		\displaystyle\prod_{s\in W}\bigl(v-s(V_1)\bigr)\\
		\hline
		\{123,231,312\}
		&
		v^3+48\sqrt{-3}
		\\
		\hline
		\{132,321,213\}
		&
		v^3-48\sqrt{-3}
		\\
		\hline
	\end{array}
	\]
	
	Their product is again
	
	\[
	\bigl(v^3+48\sqrt{-3}\bigr)
	\bigl(v^3-48\sqrt{-3}\bigr)
	=
	v^6+6912.
	\]
	
	\item \textbf{Conclusion.}
	
	This example illustrates the essence of Proposition~II.
	
	Galois describes how the minimal polynomial of a primitive element over the ground field factors when one passes to an intermediate field. The conjugates of the primitive element are partitioned according to the cosets of a subgroup.
	
	For left cosets, the resulting factors have coefficients belonging to mutually conjugate intermediate fields. For right cosets, all factors have coefficients in the same fixed field \(L^H\). When \(H\) is normal, these two descriptions coincide.
	
\end{enumerate}

\section{The closure of the substitutions of the Galois group}

We present two proofs. The second is based on Proposition~II.

\begin{enumerate}
	
	\item
	
	Let \(\{V_1,V_2,\dots,V_m\}\) be the distinct roots of \(m(x)\), the minimal polynomial of \(V_1\) over \(M\), where
	
	\[
	V_1=ax_1+bx_2+cx_3+dx_4.
	\]
	
	Suppose that, for a fixed \(1\le i\le m\),
	
	\[
	V_i=F_i(V_1),
	\]
	
	where \(F_i\in\mathbb{Q}[x]\).
	
	By Lemma~I of the \textit{Mémoire}, \(F_i(V_j)\) is likewise a root of \(m(x)\) for every \(1\le i,j\le m\). Consequently, for every \(1\le i,j\le m\), the following sets are equal:
	
	\begin{equation}
		\label{eq:gel}
		\{F_i(V_1),F_i(V_2),\dots,F_i(V_m)\}
		=
		\{V_1,V_2,\dots,V_m\}
		=
		\{F_1(V_j),F_2(V_j),\dots,F_m(V_j)\}.
	\end{equation}
	
	The roots occurring in each of these sets are distinct, since
	\(F_j(V_i)=V_{p_ip_j}\) and the values
	\(V_1,V_2,\dots,V_{24}\) are pairwise distinct.
	
	The root \(F_iF_j(V_1)\) of \(m(x)\) is equal to \(V_{p_jp_i}\). On the other hand, by the right-hand equality in (\ref{eq:gel}), there exists an index \(1\le s\le m\) such that
	
	\[
	F_iF_j(V_1)=F_s(V_1).
	\]
	
	It follows that
	
	\[
	p_jp_i=p_s.
	\]
	
	Hence the set of permutations
	
	\[
	\{p_1,p_2,\dots,p_m\}
	\]
	
	forms a group and is therefore a \textit{groupe de permutations}. Consequently, the associated set of substitutions is a \textit{groupe de substitutions}, which, by Lemma~\ref{lem:drie}, is itself a group.
	
	\item
	
	Alternatively, let \(V'\) in Proposition~II be the root
	
	\[
	V_j=F_j(V_1)
	\]
	
	of \(f(v,r)\). Proposition~II then implies that \(F_iF_j(V_1)\) is also a root of
	
	\[
	m(x)=f(v,r).
	\]
	
	The remainder of the proof is identical to the preceding argument.
	
\end{enumerate}

 \section{A simple derivation of the functions \(\theta_1,\theta_2,\theta_3,\theta_4\), and a comparison with the corresponding functions introduced by Adolf Hurwitz in his 1909 notebook}
 
 Each value \(V_i\) is associated with a permutation
 \(\pi_i=i_1i_2i_3i_4\), where
 
 \[
 V_i=ax_{i_1}+bx_{i_2}+cx_{i_3}+dx_{i_4}.
 \]
 
 For \(s=1,2,3,4\), the quantity \(x_{\pi_i(s)}\) denotes the root occurring in \(V_i\) with coefficient \(a,b,c,\) and \(d\), respectively. We define the polynomial
 
 \[
 P_j(v)
 =
 \prod_{i\neq j}(v-V_i),
 \]
 
 associated with \(V_j\). It satisfies, by means of the Kronecker delta,
 
 \[
 P_i(V_j)=\delta_{ij}\,P_i(V_i).
 \]
 
 We now define the functions \(\theta_s(v)\) by
 
 \begin{equation}
 	\label{eq:theta}
 	\theta_s(v)=
 	\frac{
 		x_{\pi_1(s)}P_1(v)+x_{\pi_2(s)}P_2(v)+\cdots+x_{\pi_{24}(s)}P_{24}(v)
 	}
 	{
 		P_1(v)+P_2(v)+\cdots+P_{24}(v)
 	}.
 \end{equation}
 
 Substituting \(v=V_i\) immediately yields
 
 \[
 \theta_s(V_i)=x_{\pi_i(s)},
 \qquad
 i=1,\dots,24,\quad s=1,\dots,4,
 \]
 
 since
 
 \[
 P_i(V_j)=\delta_{ij}P_i(V_i).
 \]
 
 It remains to prove that both the numerator and the denominator belong to \(\mathbb{Q}[v]\). Define the polynomial
 
 \begin{equation}
 	\label{eq:product}
 	h(v)=\prod_{i=1}^{24}(v-V_i).
 \end{equation}
 
 The polynomial~\eqref{eq:product} belongs to \(\mathbb{Q}[v]\) because it is symmetric in \(x_1,x_2,x_3,\) and \(x_4\). Its derivative also belongs to \(\mathbb{Q}[v]\), and this derivative is precisely the denominator of formula~\eqref{eq:theta}.
 
 Although \(a,b,c,\) and \(d\) are fixed integers, we temporarily regard them as independent real variables. Formula~\eqref{eq:product} then becomes a polynomial in the variables \(a,b,c,d,\) and \(v\). Differentiating it partially with respect to \(a,b,c,\) and \(d\) yields, up to a minus sign, exactly the numerators of formula~\eqref{eq:theta} for \(s=1,2,3,\) and \(4\). After substituting the chosen integer values of \(a,b,c,\) and \(d\), one again obtains a polynomial in \(\mathbb{Q}[v]\).
 
 It is remarkable that Adolf Hurwitz employs precisely this method in his proof of the Fundamental Theorem of Galois Theory; see his 1909 notebook \cite[pp.~153--154]{Hurdag}. It is important that all four functions \(\theta_s(v)\) have the same denominator in formula~\eqref{eq:theta}, since this is essential when Galois applies Lemma~I to formulas involving the functions \(\theta_s(v)\).
 
 One may also prove that the numerator and denominator of~\eqref{eq:theta} belong to \(\mathbb{Q}[v]\) by showing that they remain invariant under the transposition of two roots \(x_i\) and \(x_j\). We shall carry out this argument for the numerator of \(\theta_1(v)\); the preceding proof based on differentiating~\eqref{eq:product} is, of course, entirely valid. The proofs for \(\theta_2(v)\), \(\theta_3(v)\), and \(\theta_4(v)\) are analogous.
 
 Partition the \(24\) values \(V_i\) as follows. Let \(W_1\) denote the set of indices corresponding to values of the form \(ax_1+\cdots\); there are six such values. Similarly, let \(W_2\) denote the set of indices corresponding to values of the form \(ax_2+\cdots\), again consisting of six elements, and so on.
 
 Since
 
 \[
 P_j(v)
 =
 \prod_{i\neq j}(v-V_i),
 \]
 
 the numerator becomes
 
 \[
 x_1\!\left(\sum_{i\in W_1}P_i\right)
 +
 x_2\!\left(\sum_{i\in W_2}P_i\right)
 +
 x_3\!\left(\sum_{i\in W_3}P_i\right)
 +
 x_4\!\left(\sum_{i\in W_4}P_i\right).
 \]
 
 If one interchanges \(x_i\) and \(x_j\), then, apart from the exchange of \(x_i\) and \(x_j\), only the sets \(W_i\) and \(W_j\) are interchanged, so that the above sum remains unchanged. Since every permutation is a product of transpositions, it follows that the numerator is invariant under every permutation of the roots and therefore belongs to \(\mathbb{Q}[v]\).
 
 \section{The substitutions of the Galois group and automorphisms}
 \label{sec:subaut}
 
 We now prove that the substitutions belonging to the Galois group are in fact automorphisms. The setting is again
  \(
 \mathbb{Q}\subseteq\mathbb{Q}(x_1,x_2,x_3,x_4).
 \)
  Let
  \(
 \alpha=\delta(x_1,x_2,x_3,x_4)=\mathcal{F}(V_1),
 \) 
 where \(\mathcal{F}\in\mathbb{Q}[x]\) and
 \(\delta\in\mathbb{Q}[x_1,x_2,x_3,x_4]\). Then
  \[
 \alpha
 =
 \delta\!\bigl(\theta_1(V_1),\theta_2(V_1),
 \theta_3(V_1),\theta_4(V_1)\bigr)
 =
 \mathcal{F}(V_1),
 \]
 
 and therefore, by Lemma~I of the \textit{Mémoire},
  \(
 \sigma_i(\alpha)
 =
 \mathcal{F}(\sigma_i(V_1)),
 \)
  for \(i=1,2,\dots,m\), where \(g(x)\) is the minimal polynomial of \(V_1\) over \(\mathbb{Q}\), and where \(\sigma_i\) denotes the substitution
  \(
 \begin{pmatrix}
 	V_1\\
 	V_i
 \end{pmatrix}.
 \)
 
 \begin{enumerate}
 	
 	\item \textbf{Independence of the representation.}
 	
 	Suppose that
 	 	\(
 	\alpha
 	=
 	\epsilon(x_1,x_2,x_3,x_4)
 	=
 	\mathcal{G}(V_1),
 	\)
 	 	where
 	\(\mathcal{G}\in\mathbb{Q}[x]\) and
 	\(\epsilon\in\mathbb{Q}[x_1,x_2,x_3,x_4]\). Then
 	 	\(
 	\mathcal{F}(V_1)=\mathcal{G}(V_1).
 	\)
 	Lemma~I of the \textit{Mémoire} therefore implies
 	\(
 	\mathcal{F}(V_i)=\mathcal{G}(V_i),
 		i=1,\dots,m.
 	\)
 	
 	Hence the image of \(\alpha\) under \(\sigma_i\) is independent of the particular functional expression chosen for \(\alpha\) in terms of \(x_1,x_2,x_3,x_4\).
 	
 	\item \textbf{Preservation of addition and multiplication.}
 	
 	Let
 	\(
 	\alpha
 	=
 	\delta(x_1,x_2,x_3,x_4)
 	=
 	\mathcal{F}(V_1),
 	\)
 	 	and
 	 	\(
 	\beta
 	=
 	\epsilon(x_1,x_2,x_3,x_4)
 	=
 	\mathcal{G}(V_1).
 	\)
 	 	Then
 	
 	\[
 	\sigma_i(\alpha+\beta)
 	=
 	\mathcal{F}(\sigma_i(V_1))
 	+
 	\mathcal{G}(\sigma_i(V_1))
 	=
 	\sigma_i(\alpha)
 	+
 	\sigma_i(\beta),
 	\]
 	
 	for \(i=1,\dots,m\). The same argument establishes preservation of multiplication.
 	
 	\item \textbf{Injectivity and surjectivity.}
 	
 	Every substitution \(\sigma_i\) possesses an inverse. Hence every \(\sigma_i\) is both injective and surjective.
 	
 \end{enumerate}

  The converse also holds. Suppose that the substitution
  \(
 \begin{pmatrix}
 	V_1\\
 	V_i
 \end{pmatrix}
 \) 
 is an automorphism. Then, by Proposition~\ref{prop:nuttig}, \(V_i\) must be a root of the minimal polynomial of \(V_1\) over \(\mathbb{Q}\). We therefore obtain the following proposition, which establishes the connection with the modern formulation of Galois theory in terms of automorphisms.
 
 \begin{proposition}
 	\label{prop:auto}
 	The substitutions that are automorphisms are precisely the substitutions belonging to the Galois group.
 \end{proposition}
 
 We conclude with an example showing that not every substitution is an automorphism. Consider the splitting field of \(x^4-2\) over \(\mathbb{Q}\).
  Let
  \(
 L=\mathbb{Q}(\alpha,i),
 \qquad
 \alpha=\sqrt[4]{2},
 \)
  whose roots are
  \(
 x_1=\alpha,\quad
 x_2=-\alpha,\quad
 x_3=i\alpha,\quad
 x_4=-i\alpha.
 \) 
 Define
  \(
 \sigma(\alpha)=i\alpha,
 \qquad
 \sigma(i)=i,
 \)
  and
  \(
 \tau(\alpha)=\alpha,
 \quad
 \tau(i)=-i.
 \)
  Then
  \[
 \operatorname{Gal}(L/\mathbb{Q})
 =
 \{
 1,\sigma,\sigma^2,\sigma^3,
 \tau,\tau\sigma,\tau\sigma^2,\tau\sigma^3
 \}.
 \]
 
 Here we adopt the convention that, in the composition \(\tau\sigma\), the automorphism \(\sigma\) is applied first and then \(\tau\). The multiplication table is therefore represented by
 
 \[
 \begin{array}{|c|c|c|c|c|c|}
 	\hline
 	\text{Automorphism}
 	&\multicolumn{4}{c|}{\text{Substitution}}
 	&\text{One-line notation}\\
 	\hline
 	1
 	&x_1&x_2&x_3&x_4
 	&1234\\
 	\hline
 	\sigma
 	&x_3&x_4&x_2&x_1
 	&3421\\
 	\hline
 	\sigma^2
 	&x_2&x_1&x_4&x_3
 	&2143\\
 	\hline
 	\sigma^3
 	&x_4&x_3&x_1&x_2
 	&4312\\
 	\hline
 	\tau
 	&x_1&x_2&x_4&x_3
 	&1243\\
 	\hline
 	\tau\sigma
 	&x_4&x_3&x_2&x_1
 	&4321\\
 	\hline
 	\tau\sigma^2
 	&x_2&x_1&x_3&x_4
 	&2134\\
 	\hline
 	\tau\sigma^3
 	&x_3&x_4&x_1&x_2
 	&3412\\
 	\hline
 \end{array}
 \]
 
 Thus, viewed as a permutation group on the four roots,
 
 \[
 \operatorname{Gal}(L/\mathbb{Q})
 =
 \{
 1234,\,
 3421,\,
 2143,\,
 4312,\,
 1243,\,
 4321,\,
 2134,\,
 3412
 \}.
 \]
 
 In terms of \(\alpha\), the images are, for example,
 
 \[
 \sigma:
 (\alpha,-\alpha,i\alpha,-i\alpha)
 \longmapsto
 (i\alpha,-i\alpha,-\alpha,\alpha),
 \]
 
 and
 
 \[
 \tau:
 (\alpha,-\alpha,i\alpha,-i\alpha)
 \longmapsto
 (\alpha,-\alpha,-i\alpha,i\alpha).
 \]
 
 These eight automorphisms form the dihedral group
  \(
 \operatorname{Gal}(L/\mathbb{Q})
 \cong
 D_4.
 \)
  Now consider the substitutions
 \(\zeta=2314\) and
 \(\eta=3214\).
 Then
  \(
 \zeta(x_1+x_2)\neq0,
 \)
  although \(x_1+x_2=0\), while
  \(
 \eta(x_2+x_3)=0,
 \)
  although \(x_2+x_3\neq0\).
  The same is true for the substitutions belonging to the two corresponding right cosets of
 \(\operatorname{Gal}(L/\mathbb{Q})\); these are likewise not automorphisms.

\end{document}